\documentclass[11pt]{amsart}
\usepackage{amsmath,amsthm,amssymb,xypic,array}
\usepackage[T1]{fontenc}
\usepackage{amsfonts}
\usepackage{amsmath}
\usepackage{amssymb}
\usepackage{amsthm}
\usepackage{setspace}
\usepackage[cp850]{inputenc}
\usepackage{hyperref}
\usepackage{graphicx}
\usepackage{fancyhdr}
\usepackage{tikz}
\usepackage{booktabs}
\usepackage{longtable}
\usepackage{geometry}
\usetikzlibrary{trees}
\providecommand{\U}[1]{\protect\rule{.1in}{.1in}}
\providecommand{\U}[1]{\protect\rule{.1in}{.1in}}
\providecommand{\U}[1]{\protect\rule{.1in}{.1in}}
\providecommand{\U}[1]{\protect\rule{.1in}{.1in}}
\providecommand{\U}[1]{\protect\rule{.1in}{.1in}}
\input{xy}
\xyoption{all}
\def\leq{\leqslant}
\def\geq{\geqslant}

\def\bibaut#1{{\sc #1}}

\def\phi{\varphi}
\def\ro[#1]{{\textcolor{red}{#1}}}

\theoremstyle{theorem}

\newtheorem{Theorem}{Theorem}[section]

\newtheorem{Lemma}[Theorem]{Lemma}
\newtheorem{Proposition}[Theorem]{Proposition}

\theoremstyle{definition}

\newtheorem{Construction}[Theorem]{Construction}

\newtheorem{Definition}[Theorem]{Definition}
\newtheorem{Remark}[Theorem]{Remark}

\newtheorem{Example}[Theorem]{Example}

\numberwithin{equation}{section}

\newcommand{\arXiv}[1]{\href{http://arxiv.org/abs/#1}{arXiv:#1}}

\newcommand{\Aut}{\operatorname{Aut}}

\newcommand{\Pic}{\operatorname{Pic}}
\newcommand{\mult}{\operatorname{mult}}

\newcommand{\Exc}{\operatorname{Exc}}

\DeclareMathOperator{\Ker}{Ker}
\DeclareMathOperator{\im}{Im}

\newcommand{\QED}{\ifhmode\unskip\nobreak\fi\quad {\rm Q.E.D.}} 

\newcommand\Span[1]{\langle{#1}\rangle}

\newcommand{\f}{\varphi}

\renewcommand{\H}{\mathcal{H}}

\renewcommand{\P}{\mathbb{P}}

\begin{document}
\title{On the automorphisms of moduli spaces of curves}

\author[Alex Massarenti]{Alex Massarenti}
\address{\sc Alex Massarenti\\
SISSA\\
via Bonomea 265\\
34136 Trieste\\ Italy}
\email{alex.massarenti@sissa.it}

\author[Massimiliano Mella]{Massimiliano Mella}
\address{\sc Massimiliano Mella\\Dipartimento di Matematica e Informatica\\
Universit\`a di Ferrara\\
Via Machiavelli 35\\
44100 Ferrara\\ Italy}
\email{mll@unife.it}

\date{July 2013}
\thanks{Partially supported by Progetto PRIN 2010 "\textit{Geometria sulle variet\`a algebriche}" MIUR and GRIFGA}

\maketitle

\begin{abstract}
In the last years the biregular automorphisms of the Deligne-Mumford's and Hassett's compactifications of the moduli space of $n$-pointed genus $g$ smooth curves have been extensively studied by \textit{A. Bruno} and the authors.  In this paper we give a survey of these recent results and extend our techniques to some moduli spaces appearing as intermediate steps of the Kapranov's and Keel's realizations of $\overline{M}_{0,n}$, and to the degenerations of Hassett's spaces obtained by allowing zero weights. 
\end{abstract}

\tableofcontents

\section*{Introduction and Survey on the automorphisms of moduli spaces of curves}
The moduli space of $n$-pointed genus $g$ curves is a central object in algebraic geometry. The scheme $M_{g,n}$ parametrizing smooth curves has been compactified by \textit{P. Deligne} and \textit{D. Mumford} in \cite{DM} by adding Deligne-Mumford stable curves as boundary points. In \cite{Ha} \textit{B. Hassett} introduced alternative compactifications of $M_{g,n}$ by allowing the marked points to have rational weights $0< a_i \leq 1$. In the last years \textit{A. Bruno} and the authors focused on the problem of determining the biregular automorphisms of all these compactifications, see \cite{BM1}, \cite{BM2}, \cite{Ma} and \cite{MM}.\\ 
In what follows we will summarize and contextualize these results. Furthermore in Section \ref{kk} we will extend our techniques to other moduli spaces of curves, namely Hassett's spaces appearing as intermediate steps of Kapranov's Construction \ref{kblusym} and of Keel's Construction \ref{con2}. Finally, in Section \ref{zw} we will compute the automorphism groups of the degenerations of Hassett's spaces obtained by allowing some of the weights $a_i$ to be zero.
\subsubsection*{The automorphism groups of $\overline{M}_{g,n}$}
The first fundamental result about the automorphisms of moduli spaces of curves is due to \textit{H.L. Royden} \cite{Ro} and dates back to $1971$. 
\begin{Theorem}\label{thRo}
Let $M_{g,n}^{un}$ be the moduli space of genus $g$ smooth curves marked by $n$ unordered points. If $2g-2+n\geq 3$ then $M_{g,n}^{un}$ has no non-trivial automorphisms.
\end{Theorem} 
For a contextualization of this result in the Teichm\"uller-theoretic
literature we refer to \cite{Moc}. The symmetric group on $n$ elements
$S_n$ acts naturally on the moduli spaces $M_{g,n}$ and on its
Deligne-Mumford compactification $\overline{M}_{g,n}$. Therefore
$S_{n}\subseteq\Aut(\overline{M}_{g,n})$ and Theorem \ref{thRo} gave a
strong evidence for the equality to hold, see also \cite{KM} for the
genus zero case. In particular \textit{G. Farkas} \cite[Question 4.6]{Fa} asked if it is true that $\Aut(\overline{M}_{0,n})\cong S_n$ for any $n\geq 5$, and it seems that also \textit{W. Fulton} pointed to this question.\\
Farkas himself brought Kapranov's paper \cite{Ka} to the attention of the second author and suggested that the construction of $\overline{M}_{0,n}$ as a blow-up of $\mathbb{P}^{n-3}$ given by \textit{M. Kapranov} could be useful for studying $\overline{M}_{0,n}$ by techniques of projective geometry.\\
In \cite{BM1} and \cite{BM2}, \textit{A. Bruno} and the second author, thanks to Kapranov's works \cite{Ka},  managed to translate issues on the moduli space $\overline{M}_{0,n}$ in terms of classical projective geometry of $\mathbb{P}^{n-3}$. Studying linear systems on $\mathbb{P}^{n-3}$ with particular base loci they derived a theorem on the fibrations of $\overline{M}_{0,n}$.
\begin{Theorem}\cite[Theorem 2]{BM2}\label{bmf}
Let $f:\overline{M}_{0,n}\rightarrow \overline{M}_{0,r}$ be a dominant morphism with connected fibers. Then $f$ factors with a forgetful map.
\end{Theorem}   
Furthermore they realized that via this theorem on fibrations they could construct a morphism of groups between $\Aut(\overline{M}_{0,n})$ and $S_{n}$. Indeed if $\f:\overline{M}_{0,n}\to\overline{M}_{0,n}$ is an automorphism and $\pi_i:\overline{M}_{0,n}\to\overline{M}_{0,n}$ is a forgetful morphism by Theorem \ref{bmf} we have the following diagram
\[
  \begin{tikzpicture}[xscale=2.5,yscale=-1.2]
    \node (A0_0) at (0, 0) {$\overline{M}_{0,n}$};
   \node (A0_1) at (1, 0) {$\overline{M}_{0,n}$};
    \node (A1_0) at (0, 1) {$\overline{M}_{0,n-1}$};
    \node (A1_1) at (1, 1) {$\overline{M}_{0,n-1}$};
    \path (A0_0) edge [->]node [auto] {$\scriptstyle{\phi^{-1}}$} (A0_1);
   \path (A1_0) edge [->,]node [auto] {$\scriptstyle{\tilde{\phi}}$} (A1_1);
    \path (A0_1) edge [->]node [auto] {$\scriptstyle{\pi_i}$} (A1_1);
    \path (A0_0) edge [->]node [auto,swap] {$\scriptstyle{\pi_{j_i}}$} (A1_0);
  \end{tikzpicture}
  \]
where $\pi_{j_i}$ is a forgetful map. This allows us to define a surjective morphism of group
\begin{equation}\label{mor}
\begin{array}{cccc}
\chi: & \Aut(\overline{M}_{0,n}) & \longrightarrow & S_{n}\\
 & \phi & \longmapsto & \sigma_{\phi}
\end{array}
\end{equation}
where 
$$
\begin{array}{cccc}
\sigma_{\phi}: & \{1,...,n\} & \longrightarrow & \{1,...,n\}\\
 & i & \longmapsto & j_{i}
\end{array}
$$
Note that in order to have a morphism of groups we have to consider $\phi^{-1}$ instead of $\phi$. Furthermore the factorization of $\pi_i\circ\phi^{-1}$ is unique. Indeed if $\pi_i\circ\phi^{-1}$ admits two factorizations $\tilde{\phi}_{1}\circ\pi_{j}$ and $\tilde{\phi}_{2}\circ\pi_{h}$ then the equality $\tilde{\phi}_{1}\circ\pi_{j}([C,x_{1},...,x_{n}]) = \tilde{\phi}_{2}\circ\pi_{h}([C,x_{1},...,x_{n}])$ for any $[C,x_{1},...,x_{n}]\in\overline{M}_{0,n}$ implies $\tilde{\phi}_{1}([C,y_{1},...,y_{n-1}]) = \tilde{\phi}_{2}([C,y_{1},...,y_{n-1}])$ for any $[C,y_{1},...,y_{n-1}]\in\overline{M}_{0,n-1}$.
Now $\tilde{\phi}_{1} = \tilde{\phi}_{2}$ implies $\tilde{\phi}_{1}\circ\pi_{j} = \tilde{\phi}_{1}\circ\pi_{h}$ and since $\tilde{\phi}_{1}$ is an isomorphism we have $\pi_{j} = \pi_{h}$.\\ 
Once again, via the projective geometry inherited by Kapranov's construction, the kernel of $\chi$ consists of automorphism inducing on
$\P^{n-3}$ a birational self-map that stabilizes lines and
rational normal curves through $(n-1)$ fixed points. This
proves that the kernel is trivial, see the proof of
Theorem \ref{th:XY} for the details, and gives the following positive answer to \cite[Question 4.6]{Fa}.
\begin{Theorem}\cite[Theorem 3]{BM2}
The automorphism group of $\overline{M}_{0,n}$ is isomorphic to $S_{n}$ for any $n\geq 5$.
\end{Theorem}
Although a similar statement in higher genus was expected for many years the problem of computing $\Aut(\overline{M}_{g,n})$ for $g\geq 1$ was not explicitly settled. However \textit{A. Gibney}, \textit{S. Keel} and \textit{I. Morrison} gave an explicit description of the fibrations
$\overline{M}_{g,n}\rightarrow X$ of $\overline{M}_{g,n}$ on a projective variety $X$ in the case $g\geq 1$, providing an analogue of Theorem \ref{bmf}. Let $N$ be the set $\{1,...,n\}$ of the markings. If $S\subset N$ then $S^{c}$ denotes its complement.  
\begin{Theorem}\cite[Theorem 0.9]{GKM}\label{GKM}
Let $D\in \Pic(\overline{M}_{g,n})$ be a nef divisor. 
\begin{enumerate}
\item[-] If $g\geq 2$ either $D$ is the pull-back of a nef divisor on $\overline{M}_{g,n-1}$ via one of the forgetful morphisms or $D$ is big and the exceptional locus of $D$ is contained in $\partial\overline{M}_{g,n}$.
\item[-] If $g = 1$ either $D$ is the tensor product of pull-backs of nef divisors on $\overline{M}_{1,S}$ and $\overline{M}_{1,S^{c}}$ via the tautological projection for some subset $S\subseteq N$ or $D$ is big and the exceptional locus of $D$ is contained in $\partial\overline{M}_{g,n}$.
\end{enumerate}
\end{Theorem} 
An immediate consequence of Theorem \ref{GKM} is that for $g \geq 2$ any fibration of $\overline{M}_{g,n}$ to a projective variety factors through a projection to some $\overline{M}_{g,i}$ with $i < n$, while $\overline{M}_{g}$ has no non-trivial fibrations. Such a clear description of the fibrations of $\overline{M}_{g,n}$ is no longer true for $g = 1$. An explicit counterexample was given by \textit{R. Pandharipande} \cite[Example A.2]{BM2} who also observed that Theorem \ref{GKM} could be the starting point to compute the automorphism groups of $\overline{M}_{g,n}$. In order to compute $\Aut(\overline{M}_{1,n})$ the first author provided a factorization result for a particular type of fibration.
\begin{Lemma}\cite[Lemma 1.3]{Ma}\label{g1}
Let $\phi$ be an automorphism of $\overline{M}_{1,n}$. Any fibration of the type $\pi_{i}\circ\phi$ factorizes through a forgetful morphism $\pi_{j}:\overline{M}_{1,n}\rightarrow\overline{M}_{1,n-1}$.
\end{Lemma}
Thanks to Theorem \ref{GKM} and Lemma \ref{g1} in \cite{Ma} the first author constructed the analogue of the morphism \ref{mor} for $g\geq 1$ and proved the following theorem.
\begin{Theorem}\cite[Theorem 3.9]{Ma}\label{mamain}
Let $\overline{\mathcal{M}}_{g,n}$ be the moduli stack parametrizing Deligne-Mumford stable $n$-pointed genus $g$ curves, and let $\overline{M}_{g,n}$ be its coarse moduli space. If $2g-2+n\geq 3$ then 
$$\Aut(\overline{\mathcal{M}}_{g,n})\cong\Aut(\overline{M}_{g,n})\cong S_{n}.$$
For $2g-2+n < 3$ we have the following special behavior: 
\begin{itemize}
\item[-] $\Aut(\overline{M}_{1,2})\cong (\mathbb{C}^{*})^{2}$ while $\Aut(\overline{\mathcal{M}}_{1,2})$ is trivial,
\item[-] $\Aut(\overline{M}_{0,4})\cong\Aut(\overline{\mathcal{M}}_{0,4})\cong\Aut(\overline{M}_{1,1})\cong PGL(2)$ while $\Aut(\overline{\mathcal{M}}_{1,1})\cong \mathbb{C}^{*}$,
\item[-] $\Aut(\overline{M}_{g})$ and $\Aut(\overline{\mathcal{M}}_{g})$ are trivial for any $g\geq 2$, \cite[Corollary
0.12]{GKM}.  
\end{itemize}
\end{Theorem} 
The proof of Theorem \ref{mamain} is divided into two parts: the cases $2g-2+n\geq 3$ and $2g-2+n < 3$. When $2g-2+n\geq 3$ the proof uses extensively Theorem \ref{GKM}. This result, combined with the triviality of the automorphism group of the generic curve of genus $g\geq 3$, leads the first author to prove that the automorphism group of $\overline{M}_{g,1}$ is trivial for any $g\geq 3$. However any genus two curve is hyperelliptic and has a non-trivial automorphism: the hyperelliptic involution. Therefore the argument used in the case $g\geq 3$ completely fails and a different strategy is needed: the first author first proved that any automorphism of $\overline{M}_{2,1}$ preserves the boundary and then applied Theorem \ref{thRo} to conclude that $\Aut(\overline{M}_{2,1})$ is trivial. Finally he tackled the general case by induction on $n$.\\
The case $g = 1, n = 2$ requires an explicit description of the moduli space $\overline{M}_{1,2}$. In \cite[Theorem 2.3]{Ma} the first author proved that $\overline{M}_{1,2}$ is isomorphic to a weighted blow-up of $\mathbb{P}(1,2,3)$ in the point $[1:0:0]$. In particular $\overline{M}_{1,2}$ is toric. From this he derived that $\Aut(\overline{M}_{1,2})$ is isomorphic to $(\mathbb{C}^{*})^{2}$ \cite[Proposition 3.8]{Ma}.
\subsubsection*{Automorphisms of Hassett's moduli spaces}
Many of the techniques used to deal with the automorphisms of $\overline{M}_{g,n}$ apply also to moduli spaces of weighted pointed curves. These are the compactifications of $M_{g,n}$ introduced by \textit{B. Hassett} in \cite{Ha}. Hassett constructed new compactifications $\overline{\mathcal{M}}_{g,A[n]}$ of the moduli stack $\mathcal{M}_{g,n}$ and $\overline{M}_{g,A[n]}$ for the coarse moduli space by assigning rational weights $A = (a_{1},...,a_{n})$, $0< a_{i}\leq 1$ to the markings.\\  
In \cite[Section 2.1.1]{Ha} \textit{B. Hassett} considers a natural variation on the moduli problem of weighted pointed stable curves by allowing some of the marked points to have weight zero. We introduce these more general spaces because we will compute their automorphisms in Section \ref{zw}. Consider the data $(g,\tilde{A}):= (g,a_1,...,a_n)$ such that $a_{i}\in\mathbb{Q}$, $0\leq a_i\leq 1$ for any $i = 1,...,n$, and 
$$2g-2+\sum_{i=1}^{n}a_{i} > 0.$$
\begin{Definition}\label{defha}
A family of nodal curves with marked points $\pi:(C,s_{1},...,s_{n})\rightarrow S$ is stable of type $(g,\tilde{A})$ if
\begin{itemize}
\item[-] the sections $s_{1},...,s_{n}$ with positive weights lie in the smooth locus of $\pi$, and for any subset $\{s_{i_{1}},...,s_{i_{r}}\}$ with non-empty intersection we have $a_{i_{1}} +...+ a_{i_{r}} \leq 1$,
\item[-] $K_{\pi}+\sum_{i=1}^{n}a_{i}s_{i}$ is $\pi$-relatively ample. 
\end{itemize}
\end{Definition}
There exists a connected Deligne-Mumford stack $\overline{\mathcal{M}}_{g,\tilde{A}[n]}$ representing the moduli problem of pointed stable curves of type $(g,\tilde{A})$. The corresponding coarse moduli scheme $\overline{M}_{g,\tilde{A}[n]}$ is projective over $\mathbb{Z}$.\\ 
If $a_i>0$ for any $i=1,...,n$, by \cite[Theorem 3.8]{Ha} a weighted pointed stable curve admits no infinitesimal automorphisms and its infinitesimal deformation space is unobstructed of dimension $3g-3+n$. Then $\overline{\mathcal{M}}_{g,A[n]}$ is a smooth Deligne-Mumford stack of dimension $3g-3+n$. However, when some of the marked points are allowed to have weight zero even the stack $\overline{\mathcal{M}}_{g,\tilde{A}[n]}$ may be singular, see \cite[Section 2.1.1.]{Ha}.\\
Following Hassett we denote by $A$ the subset of $\tilde{A}$ containing all the positive weights so that $|\tilde{A}| = |A|+N$ where $N$ is the number of zero weights. Note that an $\tilde{A}$-stable curve is an $A$-stable curve with $N$ additional arbitrary marked points. Furthermore the points with weight zero are allowed to be singular points. So at the level of the stacks
$$\overline{\mathcal{M}}_{g,\tilde{A}[n]}\cong \underbrace{\mathcal{C}_{g,A[n]}\times_{\overline{\mathcal{M}}_{g,A[n]}}...\times_{\overline{\mathcal{M}}_{g,A[n]}}\mathcal{C}_{g,A[n]}}_{N \; \rm times}$$
where $\mathcal{C}_{g,A[n]}$ is the universal curve over $\overline{\mathcal{M}}_{g,A[n]}$.\\
In the more general setting of zero weights we still have natural morphisms between Hassett's spaces. Fixed $g,n,$ consider two collections of weight data $\tilde{A}[n],\tilde{B}[n]$ such that $a_i\geq b_i\geq 0$ for any $i = 1,...,n$. Then there exists a birational \textit{reduction morphism}
$$\rho_{\tilde{B}[n],\tilde{A}[n]}:\overline{M}_{g,\tilde{A}[n]}\rightarrow\overline{M}_{g,\tilde{B}[n]}$$
associating to a curve $[C,s_1,...,s_n]\in\overline{M}_{g,\tilde{A}[n]}$ the curve $\rho_{\tilde{B}[n],\tilde{A}[n]}([C,s_1,...,s_n])$ obtained by collapsing components of $C$ along which $K_C+b_1s_1+...+b_ns_n$ fails to be ample.\\
Furthermore, for any $g$ consider a collection of weight data $\tilde{A}[n]=(a_1,...,a_n)$ and a subset $\tilde{A}[r]:=(a_{i_{1}},...,a_{i_{r}})\subset \tilde{A}[n]$ such that $2g-2+a_{i_{1}}+...+a_{i_{r}}>0$. Then there exists a \textit{forgetful morphism} 
$$\pi_{\tilde{A}[n],\tilde{A}[r]}:\overline{M}_{g,\tilde{A}[n]}\rightarrow\overline{M}_{g,\tilde{A}[r]}$$
associating to a curve $[C,s_1,...,s_n]\in\overline{M}_{g,\tilde{A}[n]}$ the curve $\pi_{\tilde{A}[n],\tilde{A}[r]}([C,s_1,...,s_n])$ obtained by collapsing components of $C$ along which $K_C+a_{i_{1}}s_{i_{1}}+...+a_{i_{r}}s_{i_{r}}$ fails to be ample. For the details see \cite[Section 4]{Ha}.\\
Some of the spaces $\overline{M}_{0,A[n]}$ appear as intermediate steps of the Kapranov's blow-up construction of $\overline{M}_{0,n}$ \cite[Section 6.1]{Ha}. In higher genus $\overline{M}_{g,A[n]}$ may be related to the Log minimal model program on $\overline{M}_{g,n}$. Let us recall the Kapranov's construction.
\begin{Construction}\cite{Ka}\label{kblu}
Fixed $(n-1)$-points $p_{1},...,p_{n-1}\in\mathbb{P}^{n-3}$ in linear general position:
\begin{itemize}
\item[(1)] Blow-up the points $p_{1},...,p_{n-2}$, then the lines $\left\langle p_{i},p_{j}\right\rangle$ for $i,j = 1,...,n-2$,..., the $(n-5)$-planes spanned by $n-4$ of these points.  
\item[(2)] Blow-up $p_{n-1}$, the lines spanned by pairs of points including $p_{n-1}$ but not $p_{n-2}$,..., the $(n-5)$-planes spanned by $n-4$ of these points including $p_{n-1}$ but not $p_{n-2}$.\\
\vdots
\item[($r$)] Blow-up the linear spaces spanned by subsets $\{p_{n-1},p_{n-2},...,p_{n-r+1}\}$ so that the order of the blow-ups is compatible with the partial order on the subsets given by inclusion, the $(r-1)$-planes spanned by $r$ of these points including $p_{n-1},p_{n-2},...,p_{n-r+1}$ but not $p_{n-r}$,..., the $(n-5)$-planes spanned by $n-4$ of these points including $p_{n-1},p_{n-2},...,p_{n-r+1}$ but not $p_{n-r}$.\\
\vdots
\item[($n-3$)] Blow-up the linear spaces spanned by subsets $\{p_{n-1},p_{n-2},...,p_{4}\}$.
\end{itemize}
The composition of these blow-ups is the morphism $f_{n}:\overline{M}_{0,n}\rightarrow\mathbb{P}^{n-3}$ induced by the psi-class $\Psi_{n}$.
\end{Construction}
In \cite[Section 6.1]{Ha} Hassett interprets the intermediate steps of Construction \ref{kblu} as moduli spaces of weighted rational curves. Consider the weight data 
$$A_{r,s}[n]:= (\underbrace{1/(n-r-1),...,1/(n-r-1)}_{(n-r-1)\; \rm times}, s/(n-r-1), \underbrace{1,...,1}_{r\; \rm times})$$ 
for $r = 1,...,n-3$ and $s = 1,...,n-r-2$. Then the variety obtained at the $r$-th step once we finish blowing-up the subspaces spanned by subsets $S$ with $|S|\leq s+r-2$ is isomorphic to $\overline{M}_{0,A_{r,s}[n]}$.\\
To help the reader getting acquainted with Construction \ref{kblu} we develop in detail the simplest case in genus zero. 
\begin{Example} Let $n = 5$, and fix $p_{1},...,p_{4}\in\mathbb{P}^{2}$ points in general position. The Kapranov's map $f_5$ is as follows: blow-up $p_{1},p_{2},p_{3}$ and then blow-up $p_{4}$.\\
At the step $r = 1, s=1$ we get $\overline{M}_{0,A_{1,1}[n]} = \mathbb{P}^{2}$ and the weights are
$$A_{1,1}[5]:= (1/3,1/3,1/3,1/3,1).$$
While for $r = 2, s = 1$ we get $\overline{M}_{0,A_{2,1}[n]}\cong\overline{M}_{0,5}$, indeed in this case the weight data are
$$A_{2,1}[5]:= (1/2,1/2,1/2,1,1).$$
Note that as long as all the weights are strictly greater than $1/3$, the Hassett's space is isomorphic to $\overline{M}_{0,n}$ because at most two points can collide, so the only components that get contracted are rational tail components with exactly two marked points. Since these have exactly three special points they have no moduli and contracting them does not affect the coarse moduli space even though it does change the universal curve, see also \cite[Corollary 4.7]{Ha}. In our case $\overline{M}_{0,A_{2,1}[5]}\cong\overline{M}_{0,5}$.\\ 
We have only one intermediate step, namely $r = 1, s = 2$. The moduli space $\overline{M}_{0,A_{1,2}[5]}$ parametrizes weighted pointed curves with weight data 
$$A_{1,2}[5]:= (1/3,1/3,1/3,2/3,1).$$
This means that the point $p_{5}$ is allowed to collide with $p_{1},p_{2},p_{3}$ but not with $p_{4}$ which has not yet been blown-up. The Kapranov's map $f_{5}:\overline{M}_{0,5}\rightarrow\mathbb{P}^{2}$ factorizes as
  \[
  \begin{tikzpicture}[xscale=2.0,yscale=-1.5]
    \node (A0_0) at (0, 0) {$\overline{M}_{0,5}\cong\overline{M}_{0,A_{2,1}[5]}$};
    \node (A1_1) at (1, 1) {$\overline{M}_{0,A_{1,2}[5]}$};
    \node (A2_0) at (0, 2) {$\mathbb{P}^{2}\cong\overline{M}_{0,A_{1,1}[5]}$};
    \path (A1_1) edge [->]node [auto] {$\scriptstyle{\rho_{2}}$} (A2_0);
    \path (A0_0) edge [->]node [auto,swap] {$\scriptstyle{f_{5}}$} (A2_0);
    \path (A0_0) edge [->]node [auto] {$\scriptstyle{\rho_ {1}}$} (A1_1);
  \end{tikzpicture}
  \]
where $\rho_{1},\rho_{2}$ are the corresponding reduction morphisms. Let us analyze these two morphisms.
\begin{itemize}
\item[-] Given $(C,s_{1},...,s_{5})\in \overline{M}_{0,A_{2,1}[5]}$ the curve $\rho_{1}(C,s_{1},...,s_{5})$ is obtained by collapsing components of $C$ along which $K_{C}+\frac{1}{3}s_{1}+\frac{1}{3}s_{2}+\frac{1}{3}s_{3}+\frac{2}{3}s_{4}+s_{5}$ fails to be ample. So it contracts the $2$-pointed components of the following curves:
$$\includegraphics[scale=0.3]{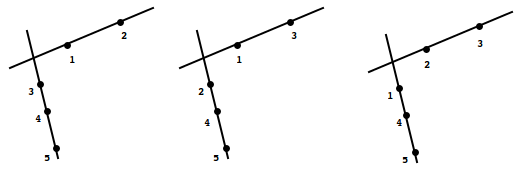}$$
along which $K_{C}+\frac{1}{3}s_{1}+\frac{1}{3}s_{2}+\frac{1}{3}s_{3}+\frac{2}{3}s_{4}+s_{5}$ is anti-ample, and the $2$-pointed components of the following curves:
$$\includegraphics[scale=0.3]{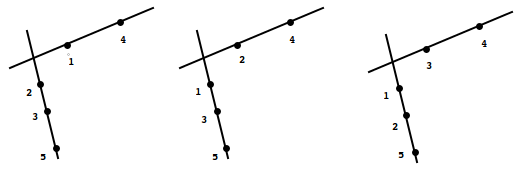}$$
along which $K_{C}+\frac{1}{3}s_{1}+\frac{1}{3}s_{2}+\frac{1}{3}s_{3}+\frac{2}{3}s_{4}+s_{5}$ is nef but not ample. However all the contracted components have exactly three special points, and therefore they do not have moduli. This affects only the universal curve but not the coarse moduli space.\\
Finally $K_{C}+\frac{1}{3}s_{1}+\frac{1}{3}s_{2}+\frac{1}{3}s_{3}+\frac{2}{3}s_{4}+s_{5}$ is nef but not ample on the $3$-pointed component of the curve
$$\includegraphics[scale=0.3]{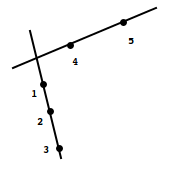}$$
In fact this corresponds to the contraction of the divisor $E_{5,4} = f_{5}^{-1}(p_{4})$.
\item[-] The morphism $\rho_{2}$ contracts the $3$-pointed components of the curves
$$\includegraphics[scale=0.3]{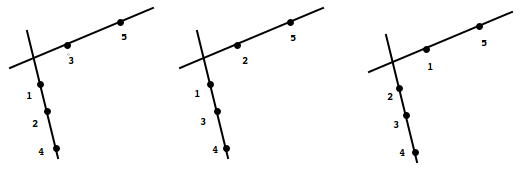}$$
along which $K_{C}+\frac{1}{3}s_{1}+\frac{1}{3}s_{2}+\frac{1}{3}s_{3}+\frac{1}{3}s_{4}+s_{5}$ has degree zero. This corresponds to the contractions of the divisors $E_{5,3} = f_{5}^{-1}(p_{3})$, $E_{5,2} = f_{5}^{-1}(p_{2})$ and $E_{5,1} = f_{5}^{-1}(p_{1})$.
\end{itemize}
\end{Example}
In the more general setting of Hassett's spaces not all forgetful maps are well defined as morphisms. However in \cite[Theorem
2.6, Proposition 2.7]{MM} the authors manage to derive a weighted version of Theorem \ref{bmf} and \ref{GKM} and thanks to these they construct a morphism of groups
\begin{equation}\label{morha}
\begin{array}{cccc}
\chi: & \Aut(\overline{M}_{g,A[n]}) & \longrightarrow & S_{r}\\
 & \phi & \longmapsto & \sigma_{\phi}
\end{array}
\end{equation}
where 
$$
\begin{array}{cccc}
\sigma_{\phi}: & \{1,...,r\} & \longrightarrow & \{1,...,r\}\\
 & i & \longmapsto & j_{i}
\end{array}
$$
and $r$ is the number of well defined forgetful maps on $\overline{M}_{g,A[n]}$. We would like to stress that in the genus zero case the authors considered only the Hassett's spaces factorizing Kapranov in the sense of the following definition.
\begin{Definition}\cite[Definition 2.1]{MM}\label{dfactkap}
We say that a Hassett's moduli space $\overline{M}_{0,A[n]}$ \textit{factors Kapranov} if there exists a morphism $\rho_2$ that makes the following diagram commutative
 \[
  \begin{tikzpicture}[xscale=3.5,yscale=-1.9]
    \node (A0_0) at (0, 0) {$\overline{M}_{0,n}$};
    \node (A1_0) at (0, 1) {$\overline{M}_{0,A[n]}$};
    \node (A1_1) at (1, 1) {$\mathbb{P}^{n-3}$};
    \path (A0_0) edge [->]node [auto,swap] {$\scriptstyle{\rho_{1}}$} (A1_0);
    \path (A1_0) edge [->]node [auto] {$\scriptstyle{\rho_{2}}$} (A1_1);
    \path (A0_0) edge [->]node [auto] {$\scriptstyle{f_{i}}$} (A1_1);
  \end{tikzpicture}
  \]
where $f_{i}$ is a Kapranov's map and $\rho_1$ is a reduction. We call such a  $\rho_2$  a \textit{Kapranov's factorization}.
\end{Definition}
However this is enough to compute the automorphisms of all intermediate steps of Construction \ref{kblu}.
\begin{Theorem}\label{authaka}
The automorphism groups of the Hassett's spaces appearing in Construction \ref{kblu} are given by
\begin{itemize}
\item[-] $\Aut(\overline{M}_{0,A_{r,s}[n]})\cong (\mathbb{C}^{*})^{n-3}\times S_{n-2},\: if \: r = 1,\: 1<s < n-3,$
\item[-] $\Aut(\overline{M}_{0,A_{r,s}[n]})\cong (\mathbb{C}^{*})^{n-3}\times S_{n-2}\times S_{2},\: if \: r = 1,\: s = n-3,$
\item[-] $\Aut(\overline{M}_{0,A_{r,s}[n]})\cong S_{n},\: if \: r \geq 2.$
\end{itemize}
\end{Theorem}

\begin{Remark}
The Hassett's space $\overline{M}_{0,A_{1,2}[5]}$ is the blow-up of $\mathbb{P}^{2}$ in three points in general position, that is a Del Pezzo surface $\mathcal{S}_{6}$ of degree $6$. By Theorem \ref{authaka} we recover the classical result on its automorphism group $\Aut(\mathcal{S}_{6})\cong (\mathbb{C}^{*})^{2}\times S_3\times S_2$. For a proof not using the theory of moduli of curves see \cite[Section 6]{DI}.\\
Furthermore, note that we are allowed to permute the points labeled by $1,2,3$ and to exchange the marked points $4,5$. However any permutation mapping $1,2$ or $3$ to $4$ or $5$ contracts a boundary divisor isomorphic to $\mathbb{P}^{1}$ to the point $\rho_{1}(E_{5,4})$, so it does not induce an automorphism. Furthermore the Cremona transformation lifts to the automorphism of $\overline{M}_{0,A_{1,2}[5]}$ corresponding to the transposition $4\leftrightarrow 5$.
\end{Remark}

\begin{Remark}
The step $r =1, s = n-3$ of Construction \ref{kblu} is the Losev-Manin's space $\overline{L}_{n-2}$, see \cite[Section 6.4]{Ha}. This space is a toric variety of dimension $n-3$. By Theorem \ref{authaka} we recover $(\mathbb{C}^{*})^{n-3}\subset\Aut(\overline{L}_{n-2})$. The automorphisms in $S_{n-2}\times S_{2}$ reflect on the toric setting as automorphisms of the fan of $\overline{L}_{n-2}$.\\
For example consider the Del Pezzo surface of degree six $\overline{M}_{0,A_{1,2}[5]}\cong\overline{L}_{3}\cong\mathcal{S}_{6}$. Let us say that $\mathcal{S}_{6}$ is the blow-up of $\mathbb{P}^{2}$ at the coordinate points $p_{1},p_{2},p_{3}$ with exceptional divisors $e_{1},e_{2},e_{3}$ and let us denote by $l_{i} = \left\langle p_{j},p_{k}\right\rangle$, $i\neq j,k$, $i = 1,2,3$, the three lines generated by $p_1,p_2,p_3$.\\
Such a surface can be realized as the complete intersection in $\mathbb{P}^{2}\times\mathbb{P}^{2}$ cut out by the equations $x_{0}y_{0} = x_{1}y_{1} = x_{2}y_{2}$. The six lines are given by $e_{i}=\{x_{j}=x_{k} = 0\}$, $l_{i}=\{y_{j}=y_{k} = 0\}$ for $i\neq j,k$, $i=1,2,3$. The torus $T = (\mathbb{C}^{*})^{3}/\mathbb{C}^{*}$ acts on $\mathbb{P}^{2}\times\mathbb{P}^{2}$ by
$$(\lambda_{0},\lambda_{1},\lambda_{2})\cdot ([x_{0}:x_{1}:x_{2}],[y_{0}:y_{1}:y_{2}]) = ([\lambda_{0}x_{0}:\lambda_{1}x_{1}:\lambda_{2}x_{2}],[\lambda_{0}^{-1}y_{0}:\lambda_{1}^{-1}y_{1}:\lambda_{2}^{-1}y_{2}]).$$
This torus action stabilizes $\mathcal{S}_{6}$. Furthermore $S_{2}$ acts on $\mathcal{S}_{6}$ by the transpositions $x_{i}\leftrightarrow y_{i}$, and $S_{3}$ acts on $\mathcal{S}_{6}$ by permuting the two sets of homogeneous coordinates separately. The action of $S_{3}$ corresponds to the permutations of the three points of $\mathbb{P}^{2}$ we are blowing-up, while the $S_{2}$-action is the switch of roles of exceptional divisors between the sets of lines $\{e_{1},e_{2},e_{3}\}$ and $\{l_{1},l_{2},l_{3}\}$. These six lines are arranged in a hexagon inside $\mathcal{S}_{6}$ 
$$\includegraphics[scale=0.3]{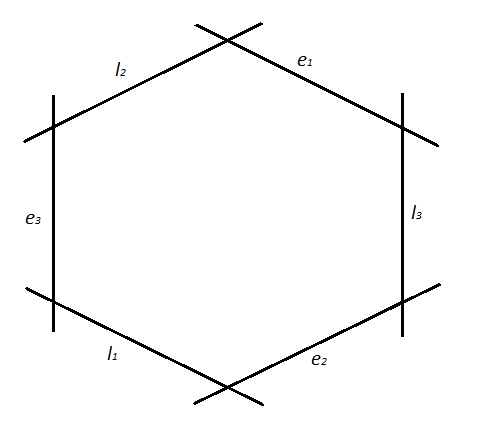}$$
which is stabilized by the action of $S_{3}\times S_{2}$. The fan of $\mathcal{S}_{6}$ is the following
$$\includegraphics[scale=0.3]{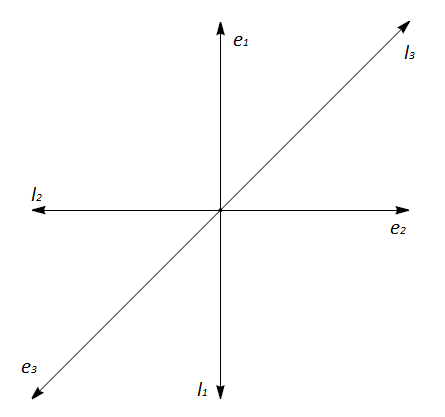}$$
where the six $1$-dimensional cones correspond to the toric divisors $e_{1},l_{3},e_{2},l_{1},e_{3}$ and $l_{2}$. It is clear from the picture that the fan has many symmetries given by permuting $\{e_{1},e_{2},e_{3}\}$, $\{l_{1},l_{2},l_{3}\}$ and switching $e_{i}$ with $l_{i}$ for $i = 1,2,3$. 
\end{Remark}

In higher genus all the forgetful maps are well defined as morphisms \cite[Lemma 3.8]{MM}. However, a transposition $i\leftrightarrow j$ of the marked points in order to induce an automorphism of $\overline{M}_{g,A[n]}$ has to preserve the weight data in a suitable sense. 

\begin{Example}
In $\overline{M}_{3,A[6]}$ with weights $(1/4,1/4,1/2,3/4,1,1)$ consider the divisor parametrizing reducible curves $C_{1}\cup C_{2}$, where $C_{1}$ has genus zero and weights $(1/4,1/4,3/4)$, and $C_{2}$ has genus three and weights $(1/2,1,1)$.
$$\includegraphics[scale=0.3]{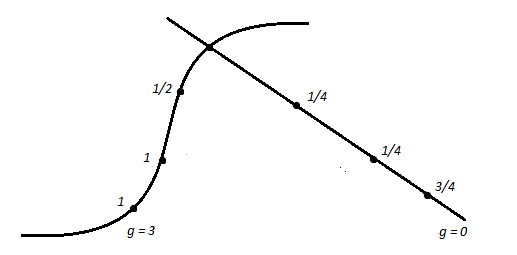}$$
After the transposition $3\leftrightarrow 4$ the genus zero component has markings $(1/4,1/4,1/1)$, so it is contracted. This means that the transposition induces a birational map
 \[
  \begin{tikzpicture}[xscale=2.5,yscale=-1.2]
    \node (A0_0) at (0, 0) {$\overline{M}_{3,A[6]}$};
    \node (A0_1) at (1, 0) {$\overline{M}_{3,A[6]}$};
    \path (A0_0) edge [->,dashed]node [auto] {$\scriptstyle{3\leftrightarrow 4}$} (A0_1);
  \end{tikzpicture}
  \]
contracting a divisor on a codimension two subscheme of $\overline{M}_{3,A[6]}$.
\end{Example}

We see that troubles come from rational tails with at least three marked points. To avoid this phenomenon the authors introduced the following definition.

\begin{Definition}\cite[Definition 3.10]{MM}\label{atrans}
A transposition $i\leftrightarrow j$ of two marked points with positive weights in $A$ is \textit{admissible} if and only if for any $h_{1},...,h_{r}\in\{1,...,n\}$, with $r\geq 2$,
$$a_{i}+\sum_{k = 1}^{r}a_{h_{k}}\leq 1 \iff a_{j}+\sum_{k = 1}^{r}a_{h_{k}}\leq 1.$$
\end{Definition}

In \cite[Lemma 3.13]{MM} the authors proved that the image of the morphism \ref{morha} is the subgroup $\mathcal{A}_{A[n]}$ of $S_n$ generated by the admissible transpositions. Finally in \cite[Theorems 3.15, 3.18]{MM} they managed to control the kernel of the morphism \ref{morha} and proved the following generalization of Theorem \ref{mamain}.
\begin{Theorem}
Let $\overline{\mathcal{M}}_{g,A[n]}$ be the Hassett's moduli stack parametrizing weighted $n$-pointed genus $g$ stable curves, and let $\overline{M}_{g,A[n]}$ be its coarse moduli space. If $g\geq 1$ and $2g-2+n\geq 3$ then
$$\Aut(\overline{\mathcal{M}}_{g,A[n]})\cong\Aut(\overline{M}_{g,A[n]})\cong \mathcal{A}_{A[n]}.$$
Furthermore 
\begin{itemize}
\item[-] $\Aut(\overline{M}_{1,A[2]})\cong (\mathbb{C}^{*})^{2}$ while $\Aut(\overline{\mathcal{M}}_{1,A[2]})$ is trivial,
\item[-] $\Aut(\overline{M}_{1,A[1]})\cong PGL(2)$ while $\Aut(\overline{\mathcal{M}}_{1,A[1]})\cong\mathbb{C}^{*}$.
\end{itemize}
\end{Theorem} 

\section{Kapranov's and Keel's spaces}\label{kk}

The Construction \ref{kblu} provides a factorization of the Kapranov's blow-up construction of $\overline{M}_{0,n}$. There are many other factorizations of the morphisms $f_i:\overline{M}_{0,n}\rightarrow\mathbb{P}^{n-3}$ as compositions of reduction morphisms. We consider other two factorizations. The first is due to Kapranov \cite[Section 6.2]{Ha}.

\begin{Construction}\label{kblusym}
Fixed $(n-1)$-points $p_{1},...,p_{n-1}\in\mathbb{P}^{n-3}$ in linear general position:
\begin{itemize}
\item[(1)] Blow-up the points $p_{1},...,p_{n-1}$, 
\item[(2)] Blow-up the strict transforms of the lines $\Span{p_{i_{1}},p_{i_{2}}}$, $i_{1},i_{2} = 1,...,n-1$,\\
\vdots
\item[($k$)] Blow-up the strict transforms of the $(k-1)$-planes $\Span{p_{i_{1}},...,p_{i_{k}}}$, $i_{1},...,i_{k} = 1,...,n-1$,\\
\vdots
\item[($n-4$)] Blow-up the strict transforms of the $(n-5)$-planes $\Span{p_{i_{1}},...,p_{i_{n-4}}}$, $i_{1},...,i_{n-4} = 1,...,n-1$.
\end{itemize}
Now, consider the Hassett's spaces $X_{k}[n]:=\overline{M}_{0,A[n]}$ for $k = 1,...,n-4$, such that
\begin{itemize}
\item[-] $a_{1}+a_{n}>1$ for $i=1,...,n-1$,
\item[-] $a_{i_{1}}+...+a_{i_{h}}\leq 1$ for each $\{i_{1},...,i_{h}\}\subset\{1,...,n-1\}$ with $r\leq n-k-2$,
\item[-] $a_{i_{1}}+...+a_{i_{h}}> 1$ for each $\{i_{1},...,i_{h}\}\subset\{1,...,n-1\}$ with $r> n-k-2$.
\end{itemize}
Then $X_{k}[n]$ is isomorphic to the variety obtained at the step $k$ of the blow-up construction.
\end{Construction}

The Hassett's spaces appearing in \cite[Section 6.3]{Ha} are strictly related to the construction of $\overline{M}_{0,n}$ provided by \textit{S. Keel} in \cite{Ke}. These spaces give another factorization of the Kapranov's map $f_i:\overline{M}_{0,n}\rightarrow\mathbb{P}^{n-3}$. 

\begin{Construction}\cite[Section 6.3]{Ha}\label{con2}
We start with the variety $Y_{0}[n] := (\mathbb{P}^{1})^{n-3}$ which can be realized as a Hassett's space $\overline{M}_{0,A[n]}$ where $A[n] = (a_1,...,a_n)$ satisfy the following conditions:
\begin{itemize}
\item[-] $a_i + a_j > 1$ where $\{i,j\}\subset \{1,2,3\}$,
\item[-] $a_i+a_{j_{1}}+...+a_{j_{r}}\leq 1$ for $i = 1,2,3$, $\{j_1,...,j_r\}\subseteq \{4,...,n\}$, with $r\geq 2$.
\end{itemize} 
Let $\Delta_{d}$ be the locus in $(\mathbb{P}^{1})^{n-3}$ where at least $n-2-d$ of the points coincide, that is the $d$-dimensional diagonal. Let $\pi_i:(\mathbb{P}^{1})^{n-3}\rightarrow\mathbb{P}^{1}$ for $i = 1,...,n-3$ be the projections, and let
$$F_{0}:= \pi_1^{-1}(0)\cup...\cup\pi_{n-3}^{-1}(0).$$
We define $F_{1}$ and $F_{\infty}$ similarly and use the same notation for proper transforms. Consider the following sequence of blow-ups
\begin{itemize}
\item[($1$)] Blow-up $\Delta_{1}\cap (F_{0}\cup F_{1}\cup F_{\infty})$.\\
\vdots
\item[($h$)] Blow-up $\Delta_{k}\cap (F_{0}\cup F_{1}\cup F_{\infty})$.\\
\vdots
\item[($n-4$)] Blow-up $\Delta_{n-4}\cap (F_{0}\cup F_{1}\cup F_{\infty})$.
\end{itemize}
The variety $Y_{h}[n]$ obtained at the step $h$ can be realized as a Hassett's space $\overline{M}_{0,A[n]}$ where the weights satisfy the following conditions:
\begin{itemize}
\item[-] $a_i+a_j > 1$ if $\{i,j\}\subset \{1,2,3\}$,
\item[-] $a_i+a_{j_1}+...+a_{j_r}\leq 1$ if $i\in \{1,2,3\}$ and $\{j_1,...,j_r\}\subset\{4,...,n\}$ with $0< r\leq n-h-3$, 
\item[-] $a_i+a_{j_1}+...+a_{j_r}> 1$ if $i\in \{1,2,3\}$ and $\{j_1,...,j_r\}\subset\{4,...,n\}$ with $r> n-h-3$.
\end{itemize}
Now, we consider another sequence of blow-ups starting from $Y_{n-4}[n]$.
\begin{itemize}
\item[($n-3$)] Blow-up $\Delta_{1}$.\\
\item[($n-2$)] Blow-up $\Delta_{2}$.\\
\vdots
\item[($2n-9$)] Blow-up $\Delta_{n-5}$.
\end{itemize}

The variety $Y_{h}[n]$ obtained at the step $h$ can be realized as a Hassett's space $\overline{M}_{0,A[n]}$ where the weights satisfy the following conditions:
\begin{itemize}
\item[-] $a_i+a_j > 1$ if $\{i,j\}\subset \{1,2,3\}$,
\item[-] $a_{j_1}+...+a_{j_r}\leq 1$ if $\{j_1,...,j_r\}\subset\{4,...,n\}$ with $0< r\leq 2n-h-7$, 
\item[-] $a_{j_1}+...+a_{j_r}> 1$ if $\{j_1,...,j_r\}\subset\{4,...,n\}$ with $r> 2n-h-7$.
\end{itemize}
\end{Construction}

\begin{Remark}
For instance taking 
$$A = (1-(n-3)\epsilon,1-(n-3)\epsilon,1-(n-3)\epsilon,\epsilon,...,\epsilon)$$
where $\epsilon$ is an arbitrarily small positive rational number, we have $\overline{M}_{0,A[n]}\cong (\mathbb{P}^{1})^{n-3}$. Note that $(\mathbb{P}^{1})^{2}$ does not admit any birational morphism to $\mathbb{P}^{2}$. However at the first step of Construction \ref{con2} we get $(\mathbb{P}^{1})^{2}$ blown-up at three points on the diagonal. Such blow-up is isomorphic to the blow-up of $\mathbb{P}^{2}$ at four general points that is $\overline{M}_{0,5}$. In the following we will prove that this fact holds also in higher dimension. More precisely the spaces $Y_{h}[n]$ factor Kapranov, in the sense of Definition \ref{dfactkap}, for any $n-4\leq h\leq 2n-9$.
\end{Remark}

Recall that the Hassett's spaces $\overline{M}_{A_{r,s}[n]}$ are the intermediate steps of Construction \ref{kblu}. 

\begin{Lemma}\label{l1}
The Hassett's spaces $Y_{h}[n]$ admit a reduction morphism to $\overline{M}_{A_{2,1}[n]}$ for any $n-4\leq h\leq 2n-9$. Furthermore $Y_{n-3}[n]\cong \overline{M}_{A_{2,2}[n]}$. 
\end{Lemma}
\begin{proof}
By construction $Y_{h}[n]$ has a reduction morphism to $Y_{n-4}[n]$ for any $h\geq n-3$. So it is enough to prove the statement for $Y_{n-4}[n]$. This variety is isomorphic to $\overline{M}_{0,A[n]}$ with
$$A[n] = (1-\epsilon,1-\epsilon,1-\epsilon,\epsilon,...,\epsilon).$$
By \cite[Corollary 4.7]{Ha}, the reduction morphism
$$A^{'}[n]:= (1,1,1-\epsilon,\epsilon,...,\epsilon)\mapsto A[n]$$
is an isomorphism. So we may proceed with $A^{'}[n]$ instead of $A[n]$. Now, take $\epsilon = \frac{1}{n-3}$ and consider the reduction morphism
$$A^{'}[n]:= (1,1,1-1/(n-3),1/(n-3),...,1/(n-3))\mapsto A_{2,1}[n].$$
The space $Y_{n-3}[n]$ can be realized as a Hassett's space with weight data
$$A[n] = (1,1,1/(n-3),1/(n-3),...,1/(n-3),2/(n-3)).$$
This is the weight data of the Hassett's spaces produced at the step $r = s = 2$ of Construction \ref{kblu}. 
\end{proof}

\begin{Proposition}\label{p1}
The Hassett's spaces $Y_{h}[n]$ factor Kapranov for any $n-4\leq h\leq 2n-9$.
\end{Proposition}
\begin{proof}
By Lemma \ref{l1} we have a reduction morphism $\rho:Y_{h}[n]\rightarrow\overline{M}_{A_{2,1}[n]}$. Since $\overline{M}_{A_{2,1}[n]}$ factors Kapranov we get the following commutative diagram
\[
  \begin{tikzpicture}[xscale=2.8,yscale=-1.2]
    \node (A0_1) at (1, 0) {$\overline{M}_{0,n}$};
    \node (A1_0) at (0, 1) {$Y_{h}[n]$};
    \node (A2_0) at (0, 2) {$\overline{M}_{0,A_{1,2}[n]}$};
    \node (A3_1) at (1, 3) {$\mathbb{P}^{n-3}$};
    \path (A0_1) edge [->] node [auto] {$\scriptstyle{\rho_{1}}$} (A2_0);
    \path (A0_1) edge [->] node [auto] {$\scriptstyle{f_{i}}$} (A3_1);
    \path (A2_0) edge [->] node [auto] {$\scriptstyle{\rho_{2}}$} (A3_1);
    \path (A0_1) edge [->,swap] node [auto] {$\scriptstyle{\psi}$} (A1_0);
    \path (A1_0) edge [->,swap] node [auto] {$\scriptstyle{\rho}$} (A2_0);
  \end{tikzpicture}
  \]
where $\psi$ is a reduction morphism. Since $\rho$, $\rho_1$ and $\psi$ are reduction morphisms $\rho_1 = \rho\circ\psi$ and $f_i = \rho_2\circ (\rho\circ\psi)$.
\end{proof}

\begin{Proposition}\label{fac1}
Assume that $\overline{M}_{0,A[r]}$ factors Kapranov. Let $f:X_{k}[n]\rightarrow\overline{M}_{0,A[r]}$ be a dominant morphism with connected fibers. Then, for any $1\leq k\leq n-4$, the morphism $f$ factors through a forgetful map. Furthermore the same holds taking $Y_{h}[n]$ instead of $X_{k}[n]$ for any $n-4\leq h\leq 2n-9$.
\end{Proposition}
\begin{proof}
Since any space $X_{k}[n]$ of Construction \ref{kblusym} factors Kapranov and by Proposition \ref{p1} any $Y_{h}[n]$ with $n-4\leq h\leq 2n-9$ factors Kapranov as well it is enough to apply \cite[Theorem 2.6]{MM}.
\end{proof}

\begin{Theorem}\label{th:XY}
The automorphism groups of the Hassett's spaces $X_{k}[n]$ for $1\leq k\leq n-4$ and $Y_{h}[n]$ for $n-4\leq h\leq 2n-9$ are given by 
$$\Aut(X_{k}[n])\cong\Aut(Y_{h}[n])\cong S_{n}.$$
\end{Theorem}
\begin{proof}
Since step $k = 1$ of Construction \ref{kblusym} we have blown-up $n-1$ points in $\mathbb{P}^{n-3}$. Furthermore the same is true for $\overline{M}_{A_{2,1}[n]}$, which is the step $r=2,s=1$ of Construction \ref{kblu}, and therefore, by Lemma \ref{l1}, for $Y_{h}[n]$ with $n-4\leq h\leq 2n-9$.\\
We may proceed by considering the spaces $X_{k}[n]$ because our proof works exactly in the same way for the spaces $Y_{h}[n]$. The key fact is that in both these classes of spaces we have blown-up $n-1$ points in linear general position in $\mathbb{P}^{n-3}$. So, by Proposition \ref{fac1}, we get a surjective morphism of groups
$$\chi:\Aut(X_{k}[n])\rightarrow S_{n}.$$
Let $\phi\in \Ker(\chi)$ be an automorphism inducing the trivial permutation. Then $\phi$ induces a birational map $\phi_{\mathcal{H}}:\mathbb{P}^{n-3}\dasharrow\mathbb{P}^{n-3}$ preserving the lines $L_{i}$ through $p_{i}$ and the rational normal curves $C$ through $p_{1},...,p_{n-1}$. Let $|\mathcal{H}|\subseteq |\mathcal{O}_{\mathbb{P}^{n-3}}(d)|$ be the linear system associated to $\phi_{\H}$. The equalities 
$$
\begin{tabular}{l}
$\deg(\phi_{\H}(L_{i})) = d - \mult_{p_{i}}\mathcal{H} = 1,$\\ 
$\deg(\phi_{\H}(C)) = (n-3)d - \sum_{i=1}^{n-1}\mult_{p_{i}}\mathcal{H} = n-3.$ 
\end{tabular} 
$$
yield $d = 1$. So $\phi_{\H}$ is an automorphism of $\mathbb{P}^{n-3}$ fixing $n-1$ points in general position, this forces $\phi_{\H} = Id$. Then $\chi$ is injective and $\Aut(\overline{M}_{0,A_{r,s}[n]})\cong S_{n}$.
\end{proof}

\section{Hassett's spaces with zero weights}\label{zw}
In this section we compute the automorphisms of the moduli spaces of weighted pointed curves $\overline{M}_{g,\tilde{A}[n]}$ of Definition \ref{defha}. Recall that we denote by $A\subseteq \tilde{A}$ the set of positive weights and by $N$ the number of zero weights.
\begin{Lemma}\label{efm}
If $g\geq 2$ then all the forgetful morphisms $\overline{M}_{g,\tilde{A}[n]}\rightarrow\overline{M}_{g,\tilde{A}[r]}$ are well defined morphisms. 
\end{Lemma}
\begin{proof}
If $g \geq 2$ we have $2g-2+a_{i_1}+...+a_{i_r}\geq 2 + a_{i_1}+...+a_{i_r}>0$. To conclude it is enough to apply \cite[Theorem 4.3]{Ha}.   
\end{proof}

\begin{Remark}
Lemma \ref{efm} does not hold for $g = 1$. For instance consider $\overline{M}_{1,\tilde{A}[2]}$ with $\tilde{A}:= (1/3,0)$. The second forgetful morphism is well defined but the first is not, being $2g-2+a_2 = 0$. To avoid this problem when $g = 1$ we will consider the Hassett's spaces $\overline{M}_{1,\tilde{A}[n]}$ such that at least two of the weights are different from zero. 
\end{Remark}

\begin{Lemma}\label{efmg1}
If $\overline{M}_{1,\tilde{A}[n]}$ is a Hassett's space with at least two weights $a_{i_1},a_{i_2}$ different from zero then all the forgetful morphisms $\overline{M}_{g,\tilde{A}[n]}\rightarrow\overline{M}_{g,\tilde{A}[n-1]}$ are well defined morphisms. 
\end{Lemma}
\begin{proof}
In any case we have $2g-2+a_{j_1}+...+a_{j_{n-1}}\geq 2g-2+a_{i_1} > 0$ or $2g-2+a_{j_1}+...+a_{j_{n-1}}\geq 2g-2+a_{i_2} > 0$. Again it is enough to apply \cite[Theorem 4.3]{Ha}.   
\end{proof}

The following proposition describes the fibrations of the Hassett's spaces $\overline{M}_{g,\tilde{A}[n]}$. Its proof derive easily from the suitable variations on the proofs of \cite[Proposition 2.7, Lemma 2.8]{MM}.

\begin{Proposition}\label{gkmha}
Let $f:\overline{M}_{g,\tilde{A}[n]}\rightarrow X$ be a dominant morphism with connected fibers.
\begin{itemize}
\item[-] If $g\geq 2$ either $f$ is of fiber type and factorizes through a forgetful morphism $\pi_{I}:\overline{M}_{g,\tilde{A}[n]}\rightarrow\overline{M}_{g,\tilde{A}[r]}$, or $f$ is birational and $\Exc(f)\subseteq\partial\overline{M}_{g,\tilde{A}[n]}$.
\item[-] If $g = 1$, $\phi$ is an automorphism of $\overline{M}_{1,\tilde{A}[n]}$ and $\tilde{A}[n]$ has at least two non-zero weights then any fibration of the type $\pi_{i}\circ\phi$ factorizes through a forgetful morphism $\pi_{j}:\overline{M}_{1,\tilde{A}[n]}\rightarrow\overline{M}_{1,\tilde{A}[n-1]}$.
\end{itemize}
\end{Proposition}

From now on we consider the case $g \geq 2$ and when $g = 1$ we restrict two Hassett's spaces having at least two non-zero weights. Let $\f:\overline{M}_{g,\tilde{A}[n]}\to\overline{M}_{g,\tilde{A}[n]}$ be an automorphism and $\pi_i:\overline{M}_{g,\tilde{A}[n]}\to\overline{M}_{g,\tilde{A}[n-1]}$ a forgetful morphism. By Proposition \ref{gkmha} we have the following diagram
\[
  \begin{tikzpicture}[xscale=2.5,yscale=-1.2]
    \node (A0_0) at (0, 0) {$\overline{M}_{g,\tilde{A}[n]}$};
   \node (A0_1) at (1, 0) {$\overline{M}_{g,\tilde{A}[n]}$};
    \node (A1_0) at (0, 1) {$\overline{M}_{g,\tilde{A}[n-1]}$};
    \node (A1_1) at (1, 1) {$\overline{M}_{g,\tilde{A}[n-1]}$};
    \path (A0_0) edge [->]node [auto] {$\scriptstyle{\phi^{-1}}$} (A0_1);
   \path (A1_0) edge [->,]node [auto] {$\scriptstyle{\tilde{\phi}}$} (A1_1);
    \path (A0_1) edge [->]node [auto] {$\scriptstyle{\pi_i}$} (A1_1);
    \path (A0_0) edge [->]node [auto,swap] {$\scriptstyle{\pi_{j_i}}$} (A1_0);
  \end{tikzpicture}
  \]
where $\pi_{j_i}$ is a forgetful map. This allows us to associate to an automorphism a permutation in $S_r$, where $r$ is the number of well defined forgetful maps, and to define a morphism of groups
$$
\begin{array}{cccc}
\chi: & \Aut(\overline{M}_{g,\tilde{A}[n]}) & \longrightarrow & S_{r}\\
 & \phi & \longmapsto & \sigma_{\phi}
\end{array}
$$
where 
$$
\begin{array}{cccc}
\sigma_{\phi}: & \{1,...,r\} & \longrightarrow & \{1,...,r\}\\
 & i & \longmapsto & j_{i}
\end{array}
$$
Note that in order to have a morphism of groups we have to consider $\phi^{-1}$ instead of $\phi$. As for the Hassett's spaces with non-zero weights the image of $\chi$ depends on the weight data. Recall that $A$ is the set of positive weights in $\tilde{A}$. By \cite[Proposition 3.12]{MM} a transposition $i\leftrightarrow j$ of two marked points with weights in $A$ induces an automorphism of $\overline{M}_{g,A[n]}$ and therefore of $\overline{M}_{g,\tilde{A}[n]}$ if and only if $i\leftrightarrow j$ is admissible.\\
Furthermore the symmetric group $S_{N}$ permuting the marked points with zero weights acts on $\overline{M}_{g,\tilde{A}[n]}$. Note that if $i\leftrightarrow j$ is a transposition switching a marked point with positive weight $a_j$ and a marked point with weight zero $a_i$ then $i\leftrightarrow j$ induces just a birational automorphism 
$$\overline{M}_{g,\tilde{A}[n]}\dasharrow \overline{M}_{g,\tilde{A}[n]}$$
because it is not defined on the loci parametrizing curves $[C,x_1,...,x_n]$ where the marked point with weight zero $x_i$ lies in the singular locus of $C$.

\begin{Example}
In $\overline{M}_{1,\tilde{A}[6]}$ with weights $(0,0,1/3,1/3,1/3,2/3)$ consider the divisor parametrizing reducible curves $C_{1}\cup C_{2}$, where $C_{1}$ has genus zero and markings $(1/3,1/3,2/3)$, and $C_{2}$ has genus one and markings $(0,0,1/3)$.
$$\includegraphics[scale=0.3]{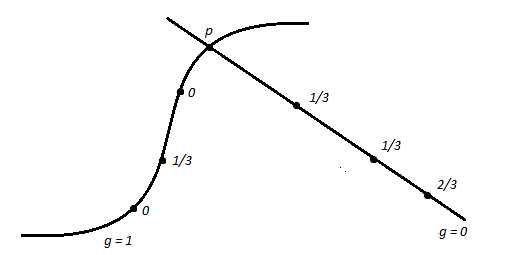}$$
After the transposition $3\leftrightarrow 6$ the genus zero component has markings $(1/3,1/3,1/3)$, so it is contracted. This means that the transposition induces a birational map
 \[
  \begin{tikzpicture}[xscale=2.5,yscale=-1.2]
    \node (A0_0) at (0, 0) {$\overline{M}_{1,\tilde{A}[6]}$};
    \node (A0_1) at (1, 0) {$\overline{M}_{1,\tilde{A}[6]}$};
    \path (A0_0) edge [->,dashed]node [auto] {$\scriptstyle{3\leftrightarrow 6}$} (A0_1);
  \end{tikzpicture}
  \]
contracting a divisor on a codimension two subscheme of $\overline{M}_{1,A[6]}$. Similarly the transposition $1\leftrightarrow 6$ does not define an automorphism because it is not defined on the locus where the first marked point $x_1$ coincides with the node $p = C_1\cap C_2$.
\end{Example}

Let us consider the subgroup $\mathcal{A}_{A[n-N]}\subseteq S_{n-N}$ generated by admissible transpositions of points with positive weights and the symmetric group $S_{N}$ permuting the marked points with zero weights. The actions of $\mathcal{A}_{A[n-N]}$ and $S_{N}$ on $\overline{M}_{g,\tilde{A}[n]}$ are independent and $\mathcal{A}_{A[n-N]}\times S_{N}\subseteq\im(\chi)$. Furthermore by \cite[Lemma 3.13]{MM} we have 
$$\im(\chi) = \mathcal{A}_{A[n-N]}\times S_{N}.$$
Finally, with the suitable variations in the proofs of \cite[Proposition 3.14]{MM} and  \cite[Theorems 3.15, 3.18]{MM} we have the following theorem. 
\begin{Theorem}
If $g\geq 2$ and if $g = 1$, $n\geq 3$, $|A|\geq 2$ then
$$\Aut(\overline{\mathcal{M}}_{g,\tilde{A}[n]})\cong\Aut(\overline{M}_{g,\tilde{A}[n]}) \cong \mathcal{A}_{A[n-N]}\times S_{N}.$$ 
\end{Theorem}

\end{document}